\documentclass[letterpaper, 10 pt, conference]{ieeeconf}
\IEEEoverridecommandlockouts
\usepackage{bm}
\usepackage{graphicx} 

\usepackage{times}
\usepackage{multicol}
\usepackage[bookmarks=true]{hyperref}
\usepackage{bm}

\usepackage{graphicx} 
\usepackage{float}
\usepackage{subfig}
\usepackage{tabularx}
\usepackage{epstopdf}

\usepackage{amssymb}
\usepackage{latexsym}
\usepackage{amsmath}
\usepackage{tensor}
\usepackage{xfrac}

\usepackage{enumerate}

\usepackage{algorithm}
\usepackage{algpseudocode}

\newcommand{\beq}{\begin{equation}}
\newcommand{\eeq}{\end{equation}}
\newcommand{\beqr}{\begin{equation}\begin{array}{l}}
\newcommand{\eeqr}{\end{array}\end{equation}}
\newcommand{\beqa}{\begin{eqnarray}}
\newcommand{\eeqa}{\end{eqnarray}}

\newtheorem{theorem}{Theorem}[section]

\def\R{\mathbb{R}}
\def\F{\mathcal{F}}

\def\P{\mathsf{P}}

\def\ds{\mathrm{d}s}
\def\dX{\mathrm{d}X}

\def\dY{\mathrm{d}Y}
\def\dW{\mathrm{d}W}
\def\Q{\mathsf{Q}}
\def\pR{\mathsf{R}}

\def\fP{{\protect \overrightarrow{\mathsf{P}}}}

\def\eps{\varepsilon}
\def\argmin{\operatornamewithlimits{arg\,min}}

\def\tr{\operatorname{tr}}

\def\wh{\widehat}

\def\wt{\widetilde}
\def\E{\mathbf{E}}

\long\def\IGNORE#1{}

\usepackage[normalem]{ulem}

\title{\LARGE \bf 
Forward-Backward Rapidly-Exploring Random Trees for\\ Stochastic Optimal Control
} 

\author{Kelsey P. Hawkins, Ali Pakniyat, Evangelos Theodorou, Panagiotis Tsiotras
\thanks{All authors are with the Georgia Institute of Technology, Atlanta, Georgia 30332--0250.
        Contact at {\tt\small kphawkins, pakniyat, evangelos.theodorou, tsiotras @gatech.edu}}%
\thanks{Support for this work has been provided by NSF award IIS-2008686.}%
}

\begin{document}

\maketitle
\thispagestyle{empty}
\pagestyle{empty}

\begin{abstract}
    We propose a numerical method for the computation of the forward-backward stochastic differential equations \mbox{(FBSDE)} appearing in the Feynman-Kac representation of the value function in stochastic optimal control problems. By the use of the Girsanov change of probability measures, it is demonstrated how a rapidly-exploring random tree (RRT) method can be utilized for the forward integration pass, as long as the controlled drift terms are appropriately compensated in the backward integration pass. Subsequently, a numerical approximation of the value function is proposed by solving a series of function approximation problems backwards in time along the edges of the constructed RRT. Moreover, a local entropy-weighted least squares Monte Carlo (LSMC) method is developed to concentrate function approximation accuracy in regions most likely to be visited by optimally controlled trajectories. The results of the proposed methodology are demonstrated on linear and nonlinear stochastic optimal control problems with non-quadratic running
    costs, which reveal significant convergence improvements over previous FBSDE-based numerical solution methods.
\end{abstract}

\section{INTRODUCTION}

The Feynman-Kac representation theory and its associated forward-backward stochastic differential equations \mbox{(FBSDEs)} has been gaining traction as a framework to solve nonlinear stochastic control problems, including optimal control problems with quadratic cost \cite{exarchos2018stochastic}, minimum-fuel ($L_1$-running cost) problems \cite{Exarchos2018, Exarchos2018a}, differential games \cite{Exarchos2016}, and reachability problems \cite{exarchos2018stochastic,mete2002stochastic}. FBSDE-based numerical methods have also received interest from the mathematical finance community \cite{Bender2007,Longstaff2001,ma2007forward}. 
Although initial results demonstrate promise in terms of flexibility and theoretical validity, numerical algorithms which leverage this theory have not yet matured. For even modest problems, state-of-the-art algorithms often have issues with slow and unstable convergence to the optimal policy. Producing more robust numerical methods is critical for the broader adoption of FBSDE methods for real-world tasks.

FBSDE numerical solution methods broadly consist of two steps, a forward pass, which generates Monte Carlo samples of the forward stochastic process, and a backward pass, which iteratively approximates the value function backwards in time. Typically, FBSDE methods perform this approximation using a least-squares Monte Carlo (LSMC) scheme, which implicitly solves the backward SDE with parametric function approximation \cite{Longstaff2001}. The approximate value function fit in the backward pass is then often used to improve sampling in an updated forward pass, leading to an iterative algorithm which, ideally, improves the approximation till convergence. Although FBSDE methods share a distinct similarity to differential dynamic programming (DDP) techniques \cite{jacobson1970differential,theodorou2010stochastic,NIPS2007_3297}, DDP is generally less flexible. For most DDP applications, a strictly positive definite running cost with respect to the control is required for convergence \cite[Section 2.2.3]{tassa2011theory}. Furthermore, in DDP, the computation of first and second order derivatives of both dynamics and costs is necessary for the backward pass, making it challenging to apply this approach to problems where these derivatives are not known analytically. In contrast, FBSDE techniques only require a good fit of the value function and the evaluation of the gradient of this value function to obtain the optimal control.

The flexibility of Feynman-Kac-based FBSDE algorithms stems from the intrinsic relationship between the solution of a broad class of second-order parabolic or elliptic PDEs to the solution of FBSDEs (see, e.g., \cite[Chapter 7]{yong1999stochastic}), brought to prominence in \cite{Pardoux1990,peng1993backward,el1997backward}. Both Hamilton-Jacobi-Bellman (HJB) and Hamilton-Jacobi-Isaacs (HJI) second order PDEs, utilized for solving stochastic optimal control and stochastic differential game equations respectively, can thus be solved via FBSDE methods, even when the costs and dynamics are nonlinear. This provides an alternative to the direct solution of PDEs, typically solved using grid-based methods such as the Level Set Toolbox \cite{Mitchell2004}, known for poor scaling in high dimensional state spaces ($n \geq 4$). 

The primary advantage of Feynman-Kac-based FBSDE methods is that they produce an unbiased estimator for the value function associated with the HJB equations. However, 
a na\"\i ve application of the theory leads to estimators with high variance by producing sample trajectories away from the optimal ones. Recent work has shown that Girsanov's theorem can be used to change the sampling measure of the forward pass without adding intrinsic bias to the estimator \cite{exarchos2018stochastic,Exarchos2018,Exarchos2018a}. That is, a change over probability spaces corresponds to the introduction of a drift to the forward SDE that can be employed to modify the sampling in the forward pass; this, in turn, requires appropriate accommodation of the change of measures in the backward pass.

In this work we expand upon the above results, by showing that the forward sampling measure can be modified at will, which enables us to incorporate methods from other domains, namely, rapidly-exploring random trees (RRTs)  (see, e.g., \cite{lavalle2001randomized} and the recent survey in \cite{noreen2016optimal}), in order to more efficiently explore the state space in the forward pass. RRTs are frequently applied to reachability-type motion planning problems, biasing the samples towards regions of the state space that have low density. Using RRTs in the forward sampling allows us to spread samples evenly over the reachable state space, increasing the likelihood that near-optimal samples are well-represented in the forward pass sample distribution. By sampling more efficiently and relying less on incremental approximations of the value function to guide our search, we can achieve faster and more robust convergence than previous FBSDE methods.
In the backward pass, we take advantage of the path-integrated running costs and estimates of
    the value function to produce a heuristic which weighs paths in the function approximation
    according to a local-entropy measure-theoretic optimization.
Although local-entropy path integral theory and RRTs have been used together in \cite{arslan2014}, called PI-RRT, this method is more closely related to the path-integral approach to control \cite{theodorou2010stochastic}.
Our method similarly performs forward passes to broadly sample the state space, but follows them with backward passes to obtain approximations for the value functions, and consequently obtain closed loop policies over the full horizon.

The primary contributions of this paper are as follows:
\begin{itemize}
    \item Providing the theoretical basis for the use of McKean-Markov branched sampling in the forward pass of FBSDE techniques.
    \item Introducing an RRT-inspired algorithm for sampling the forward SDE.
    \item Presenting a technique for concentrating value function approximation accuracy in regions containing optimal trajectories.
    \item Proposing an iterative numerical method for the purpose of approximating the optimal value function and its policy.
\end{itemize}
We call the proposed method forward-backward rapidly exploring random trees (FBRRT). After we describe the approach in both theory and numerical implementation, we apply FBRRT to two problems, comparing it to \cite{Exarchos2018}, and demonstrating its ability to solve nonlinear stochastic optimal control problems with non-quadratic running costs.

\section{THE HAMILTON-JACOBI EQUATION AND ON-POLICY VALUE FUNCTION} 


Let ${(\Omega, \F, \{\F_t\}_{t \in [0,T]}, \Q)}$, be a complete, filtered probability space, on which $W_s^\Q$ is an $n$-dimensional standard Brownian (Wiener) process with respect to the probability measure $\Q$ and adapted to the filtration $\{\F_t\}_{t \in [0,T]}$. Consider a stochastic system whose dynamics are governed by
\begin{align}
    \dX_s &= f(s,X_s,u_s)\, \ds + \sigma(s,X_s) \, \dW^\Q_s \text{,} & X_0 &= x_0 \text{,}
    \label{eq:SOCdyn}
\end{align}
where $X_s$ is a $\F_s$-progressively measurable state process on the interval $s \in [0,T]$, taking values in $\R^n$, $u_{[0,T]}$ is a progressively measurable input process on the same interval, taking values in the compact set $U \subseteq \R^m$, and $f: [0,T] \times \R^n \times U \rightarrow \R^n$, $\sigma: [0,T] \times \R^n \rightarrow \R^{n \times n}$ are the Markovian drift and diffusion functions respectively. The cost associated with a given control signal is
\begin{align}
    S_t[u_{[t,T]}] &:= \int^T_t \ell(s,X_s,u_s) \, \ds + g(X_T) \text{,} 
\end{align}
where $\ell: [0,T] \times \R^n \times U \rightarrow \R^+$ is the running cost, and \mbox{$g: \R^n \rightarrow \R^+$} is the terminal cost. Let membership of a function in $C^{l,k}_b$ denote that the function and its partial derivatives in $t$ of order $\leq l$ and in $x$ of order $\leq k$ are continuous and bounded on the domain. The membership in $C^{k}_b$ is defined similarly. We assume the functions $f$, $\sigma$, $a := \sigma \sigma^\top$ and $\ell$ belong to $C^{1,2}_b$, that $g \in C^3_b$, and that $\sigma^{-1}$ exists and is uniformly bounded on its domain.

The stochastic optimal control (SOC) problem is to determine the value function $V^*: [0,T] \times \R^n \rightarrow \R^+$ defined as
\begin{align}
    V^*(t,x) &= \inf_{u_{[t,T]}}  \E_{{\Q}}^{t,x}\big [\, S_t[u_{[t,T]} ] \,\big ] 
    \tag{SOC} \label{eq:SOC} \text{,} 
\end{align}
where $\E_{\Q}^{t,x}[\cdot] := \E_{\Q}[\cdot| X_t = x]$ denotes the conditional expectation  given $X_t = x_t$ under the probability measure~$\Q$. 
    
Under mild regularity assumptions, in particular that $\sigma \sigma^\top$ is uniformly positive definite, there exists a unique classical solution $V^* \in C^{1,2}_b$ to the Hamilton-Jacobi-Bellman PDE, as well as a (not necessarily unique) optimal Markov control policy $\pi^*$, which satisfies the inclusion
\begin{align}
    \pi^*(s,x) \in \argmin_{u \in U} \{ \ell(s,x,u) + f(s,x,u)^\top \partial_x V^*(s,x) \} \text{,}
    \label{eq:optpoli}
\end{align}
with the property that $V^*(t,x) = \E_{{\Q}}^{t,x}[\, S_t[\pi^*] \,]$, where $\partial_x V^*$ is the partial derivative of $V^*$ with respect to state~$x$ \cite[Chapter~4, Theorems~4.2 and 4.4, and Chapter~6, Theorem~6.2]{fleming1976deterministic}.
    



In this paper, instead of a direct solution of the HJB PDE, we work with a class of generic Markov policies $\mu : [0,T] \times \R^n \rightarrow U$ and their associated value functions $V^\mu$, and use iterative methods to approximate $V^*$ and $\pi^*$. 
The on-policy value function is defined as 
\begin{align}
    \begin{aligned}
        V^\mu(t,x) &= \E_{{\Q}}^{t,x}[\, S^\mu_t \,] \text{,} \\
        S^\mu_t &:= \int^T_t \ell^\mu_s \, \ds + g(X_T) \text{,} 
    \end{aligned} \label{eq:onpolicyvf}
\end{align}
with the process $X_s$ satisfying the forward SDE (FSDE)
\begin{align}
    \dX_s &= f^\mu_s \, \ds + \sigma_s \, \dW^\Q_s, \quad X_t = x \text{,}
    \label{eq:fsdeorig} 
\end{align}
where, for brevity of exposition, we define
\begin{align*}
    f^\mu_s &:= f(s,X_s,\mu(s,X_s)) \text{,}
\end{align*}
and similarly for $\ell$, $\sigma$.
We call $\mu$ an admissible Markov policy if it is Borel-measurable and its associated $V^\mu$ is the unique classic solution to the Hamilton-Jacobi PDE
\begin{gather}
    \partial_t V^\mu + \frac{1}{2} \tr[\sigma \sigma^\top \partial_{xx} V^\mu] 
    + (\partial_{x} V^\mu)^\top f^\mu + \ell^\mu \big |_{t,x} = 0 \text{,} \nonumber \\
    V^\mu(T,x) = g(x) \text{,} \tag{HJ} \label{eq:hjpde}
\end{gather}
for $(t,x) \in [0,T) \times \R^n$, where $\partial_t$ and $\partial_x$ are the partial derivative operators in $t$ and $x$, and $\partial_{xx}$ is the Hessian in $x$. Hence, the optimal control problem is expressed as $V^{*} = \min V^{\mu}$ over all $\mu$ such that \eqref{eq:hjpde}~holds.
Since the boundedness of $\sigma^{-1}$ makes the PDE non-degenerate parabolic, a sufficient, but not necessarily tight, condition guaranteeing existence of the classical solution is if $f^\mu$ and $\ell^\mu$ are in $C^{1,2}_b$ \cite[p. 156; Chapter~3, Theorem~4.2, Theorem~4.4]{fleming2006controlled}.
The same reference guarantees that \mbox{$V^* \equiv V^{\pi^*}$}.

\section{FEYNMAN-KAC-GIRSANOV FBSDE REPRESENTATION}

\subsection{On-Policy FBSDEs}
The positivity of $\sigma\sigma^{\top}$ yields that \eqref{eq:hjpde} is a parabolic PDE and, hence, by the Feynman-Kac Theorem (see, e.g. \cite{peng1991probabilistic}) it is linked to to the solution $(X_s,Y_s,Z_s)$ of the pair of FBSDEs composed of the FSDE \eqref{eq:fsdeorig} and the backward SDE (BSDE)
\begin{align}
    \dY_s &= -\ell^\mu_s \, \ds + Z^{\top}_s \, \dW_s^{\Q}, & Y_T &= g(X_T)  \text{,} 
    \label{eq:bsdeorig}
\end{align}
    where $Y_s$ and $Z_s$ are, respectively, $1$ and $n$-dimensional adapted processes.
\begin{theorem}[Feynman-Kac Representation] \label{thm:continfk}
    For the solution 
    $(X_s,Y_s,Z_s)$ to the \mbox{FBSDE} characterized by
    \eqref{eq:fsdeorig} and \eqref{eq:bsdeorig},
    it holds that
\begin{gather}
    \begin{aligned}
        Y_s &= V^\mu(s,X_s) \text{,} & s \in [0,T] \text{,} \\
        Z_s &= \sigma_s^\top \partial_x V^\mu(s,X_s) \text{,} & \text{a.e.} \;  s \in [0,T] \text{,}
    \end{aligned} \label{eq:yzv}
\end{gather}
    $\Q$-almost surely (a.s.), and, in particular,
\begin{align}
    Y_t &= \E_{\Q}[\wh{Y}_{t,\tau}| X_t] = V^\mu(t,X_t), 
    & \Q\text{-a.s.} \text{,}
    \label{eq:thmcontinfk}
\end{align}
    for $0 \leq t \leq \tau \leq T$ where
\begin{align}
    \wh{Y}_{t,\tau} &:= Y_\tau + \int_t^\tau \ell^\mu_s \ds \text{.} 
    \label{eq:bsdediff}
\end{align}
\hfill $\square$
\end{theorem}

\begin{proof}
    Equations \eqref{eq:yzv} are due directly to
    \cite[Chapter~7, Theorem~4.5, (4.29)]{yong1999stochastic}.
    From the definition of It\^o integrals, we have
\begin{align}
    \wh{Y}_{t,\tau} &= Y_t 
        - \int_t^\tau Z^{\top}_s \dW_s^{\Q}  \text{.} 
\end{align}
    Taking the conditional expectation of both sides yields
\begin{align}
    \E_{\Q}[\wh{Y}_{t,\tau}|\F_t] &= \E_{\Q}[Y_t|\F_t] = Y_t \text{,}
\end{align}
    noting that the last term drops out due to the property of the It\^o integral
    \cite[p.~34, (5.26)]{yong1999stochastic}, and $Y_t$ passes through the conditional
    expectation because it is $\F_t$ measurable.
    Equation \eqref{eq:bsdediff} is a direct consequence of the definition of It\^o integrals,
    \cite[p.~33, (5.23)]{yong1999stochastic}.
\end{proof}


\subsection{Off-Policy FBSDEs}
\label{sec:OffPolicy}

Consider, contrary to the on-policy FBSDEs, the off-policy \textit{drifted FBSDEs}
\begin{align}
    \dX_s &= K_s \, \ds + \sigma_s \, \dW^\P_s, & X_0 &= x_0 \text{,}
    \label{eq:driftfsde} \\
    \dY_s &= -(\ell^\mu_s + Z^{\top}_s D_s) \, \ds + Z^{\top}_s \dW_s^{\P}, & Y_T &= g(X_T)  \text{,}
    \label{eq:driftbsde}
\end{align}
with
\begin{align}
    D_s &:= \sigma_s^{-1} (f^\mu_s - K_s) \text{,}
\end{align} 
where $K_s$, an arbitrary $\F_s$-progressively measurable and bounded process 
satisfying the smoothness conditions of \cite[Chapter~1, Theorem~6.16]{yong1999stochastic}, $\P$ the new probability measure associated with $K_s$ and $W^\P_s$ a Brownian process over the new, complete, filtered probability space ${(\Omega, \F, \{\F_t\}_{t \in [0,T]}, \P)}$. 
\begin{theorem} \label{thm:driftfbsde}
For the solution 
    $(X_s,Y_s,Z_s)$ to the \mbox{FBSDE} characterized by
    \eqref{eq:driftfsde} and \eqref{eq:driftbsde},
    it holds that
\begin{gather}
    \begin{aligned}
        Y_s &= V^\mu(s,X_s) \text{,} & s \in [0,T] \text{,} \\
        Z_s &= \sigma_s^\top \partial_x V^\mu(s,X_s) \text{,} & \text{a.e.} \;  s \in [0,T] \text{,}
    \end{aligned} \label{eq:yzv2}
\end{gather}
    $\P$-a.s., and in particular,
\begin{align}
    Y_t &= \E_{\P}[\wh{Y}_{t,\tau}| X_t] = V^\mu(t,X_t), 
    & \P\text{-a.s.} \text{,}
    \label{eq:thmcontinfk2}
\end{align}
    where
\begin{align}
    \wh{Y}_{t,\tau} &:= Y_\tau + \int_t^\tau (\ell^\mu_s + Z^{\top}_s D_s) \, \ds \text{.} 
    \label{eq:bsdediff2}
\end{align}
\hfill $\square$
\end{theorem}

\begin{proof}
    Apply Girsanov's theorem  
    to both \eqref{eq:driftfsde} and \eqref{eq:driftbsde},
    where the Brownian process $W^\Q_s$ is defined as
    $\mathrm{d} W^\Q_s := \mathrm{d} W^\P_s - D_s \mathrm{d} s$ and the Radon-Nikodym
    derivative is defined according to \cite[Chapter 5, Theorem 10.1]{fleming1976deterministic}.
    Further, the theorem guarantees that 
    $\P$ and $\Q$ are equivalent measures in a measure-theoretic sense.
Since \eqref{eq:yzv} holds $\Q$-a.s., 
    there exists an $N \in \F$, $\Q(N) = 0$, such that $E^\mathsf{C} \subseteq N$,
    where $E := \{ \omega \in \Omega : \text{\eqref{eq:yzv} holds} \}$.
It subsequently follows from the definition of absolute continuity that $\P(N) = 0$,
    so \eqref{eq:yzv} holds $\P$-a.s. as well.
The rest follows similarly to Theorem~\ref{thm:continfk}
\end{proof}

We can interpret this result in the following sense.
As long as the diffusion function $\sigma$ is the same as in the on-policy formulation, we can pick an arbitrary process $K_s$ to be the drift term which generates a distribution for the forward process $X_s$ in the corresponding measure $\P$. The BSDE yields an expression for $Y_t$ using the same process $W^{\P}_s$ as used in the FSDE. The term $Z^{\top}_s D_s$ acts as a correction in the BSDE to compensate for changing the drift of the FSDE. We can then use the relationship \eqref{eq:thmcontinfk2} to solve for the value function $V^\mu$, whose conditional expectation can be evaluated in $\P$. Although used in the analytic construction of the value function, the measure $\Q$ does not require approximation to solve for the value function.

It should be highlighted that $K_s$ need not be a deterministic function of the random variable $X_s$, as is the case with $f^\mu_s$. For instance, it can be selected as the function \mbox{$K_s(\omega) = h(s,X_s(\omega),\omega)$} for some appropriate function $h$, producing a non-trivial joint distribution for the random variables $(X_t,K_t)$.

A remarkable feature of both the on- and off-policy FBSDEs is that the forward 
    pass is decoupled from the backward pass, that is, the evolution of the forward SDE 
    does not explicitly depend on $Y_s$ or $Z_s$
    (whereas in the Stochastic Maximum Principle formulations 
    (see, e.g., \cite[Chapter 3]{yong1999stochastic}) the decoupling is irremovable). 
This feature forms the basis of FBSDE numerical investigations of stochastic optimal control  
    \cite{Bender2010, exarchos2018stochastic}, but the significant difference of 
    Theorem~\ref{thm:continfk} in comparison to those results is that the focus is shifted here 
    from the solution of the HJB equation towards the broader class of functions satisfying 
    the \eqref{eq:hjpde}.
This provides a stronger case for policy iteration methodologies, because the theory does not
    require or expect $\mu$ to be an optimal policy, as is in 
    \cite{Bender2010, exarchos2018stochastic}.
Although not evaluated in this work, $\mu$ can be chosen according to design specifications
    other than estimating the optimal policy, such as to ensure the current policy
    is close to the previously estimated policy.

\subsection{Local Entropy Weighing}

As discussed in Section \ref{sec:OffPolicy}, the disentanglement of the forward sampling from the backward function approximation provides the opportunity to employ broad sampling schemes to cover the state space with potential paths. However, fitting a value function broadly to a wide support distribution might degrade the quality of the function approximation since high accuracy of function approximation is more in demand in those parts of the state space in proximity to optimal trajectories. Once forward sampling has been performed and some parts of the value function have been approximated, we can begin forming a heuristic in which sample paths closer to optimal trajectories are weighted more to concentrate value function approximation accuracy in those regions.

To this end, we propose using a bounded heuristic random variable $\rho_t$ 
to produce a new measure $\pR_t$, the weighted counterpart to $\P_t$, where the subscript refers to the restriction of $\P$ to $\F_t$. In order to avoid underdetermination of the regression by concentration over a single or few samples, we select $\pR_t$ as 
\begin{align}
    \pR_t \in \argmin_{\pR_t} \big \{ \E_{\pR_t}[\rho_t] + \lambda \mathcal{H}(\pR_t \| \P_t) \big \} \text{,} \label{eq:MinRelativeEntropyMeasure}
\end{align}
with $\lambda > 0$ a tuning variable 
and 
\begin{align}
    \mathcal{H}(\pR_t \| \P_t) = \E_{\pR_t}  \bigg[ \log \bigg ( 
    \frac{\mathrm{d} \pR_t}{\mathrm{d} \P_t} \bigg )  \bigg] \text{,}
\end{align}
is the relative entropy of $\pR_t$ which takes its minimum value when $\pR_t = \P_t$, the distribution in which all sampled paths have equal weight. 


The minimizer \eqref{eq:MinRelativeEntropyMeasure}, which balances between the value of $\rho$ and the relative entropy of its induced measure, has a solution of $\pR^*_t$ determined \cite[p.~2]{theodorou2012relative} as
\begin{align}
    \mathrm{d} \pR^*_t 
    &= \Theta_t 
    \mathrm{d} \P_t  \text{,} \label{eq:changemeas} \\
    \Theta_t &:= \frac{ \exp(-\sfrac{1}{\lambda} \rho_t) }{ \E_{\P_t}[\exp(-\sfrac{1}{\lambda} \rho_t)] } 
    \label{eq:thetadef}
    \text{.}
\end{align}

Henceforth, we let $\pR_t$ refer to this minimizer $\pR^*_t$. In the numerical approximation of this heuristic we can interpret the weights as a \textit{softmin} operation over paths according to the heuristic, a method often used in deep learning literature~\cite{Goodfellow-et-al-2016}.

\begin{theorem}\label{thm:weight}
Assume $\rho_\tau$ is selected such that $W^\P_s$ is Brownian on the interval $[t,\tau]$
    in the induced measure $\pR_\tau$.
It holds that
\begin{align}
    Y_t &= \E_{\P_\tau}[\wh{Y}_{t,\tau}| X_t] = V^\mu(t,X_t), 
    & \pR_\tau\text{-a.s.} \text{,}
    \label{eq:thmcontinfk3}
\end{align}
    where $\wh{Y}_{t,\tau}$ is defined in \eqref{eq:bsdediff2}.
Furthermore, the minimizer $\phi^*$ of the optimization
\begin{align}
    &\inf_{\phi \in L_2} 
    \E_{\pR_\tau}[(\wh{Y}_{t,\tau} - \phi(X_t))^2] \nonumber \\
    &\quad = \inf_{\phi \in L_2} 
    \E_{\P_\tau}[\Theta^{\pR|\P}_\tau (\wh{Y}_{t,\tau} - \phi(X_t))^2] \text{,}
    \label{eq:thmcontinfk4}
\end{align}
    over $X_t$-measurable square integrable variables $\phi(X_t)$ coincides with
    the value function $\phi^*(X_t) = V^\mu(t,X_t)$.
\hfill $\square$
\end{theorem}

\begin{proof}
First, note that $X_s$, $Y_s$, $Z_s$, and $\wh{Y}_{t,\tau}$ are $\F_\tau$-measurable for $s, t \in [0,\tau]$.
    Thus, restricting $\P$ to $\F_\tau$ in Theorem~\ref{thm:driftfbsde}, producing $\P_\tau$, 
    results in the same assumptions for those variables.
    Since $\rho_\tau$ is bounded, $\Theta_\tau > 0$ $\P$-a.s..
    Further, we have $\E_{\P_\tau}[\Theta_\tau] = 1$, so the variable is normalized.
    It is easy to see that this guarantees that $\pR_\tau$ is a probability measure and
    the measures $\pR_\tau$ and $\P_\tau$ are equivalent.
    It follows that \eqref{eq:thmcontinfk3} holds.
    Equation \eqref{eq:thmcontinfk4} is a result of the $L_2$-projective properties
    of conditional expectation \cite{resnick2003probability}
    and then a change of measure with \eqref{eq:changemeas}.
\end{proof}
In the following section, we evaluate the minimization of the right hand side of 
    \eqref{eq:thmcontinfk4} over parameterized value function models 
    to obtain an estimate of the value function.

\begin{figure}
    \centering
    \subfloat[Optimal Distribution\label{fig:distopt}]{
        \includegraphics[width=0.47\linewidth]{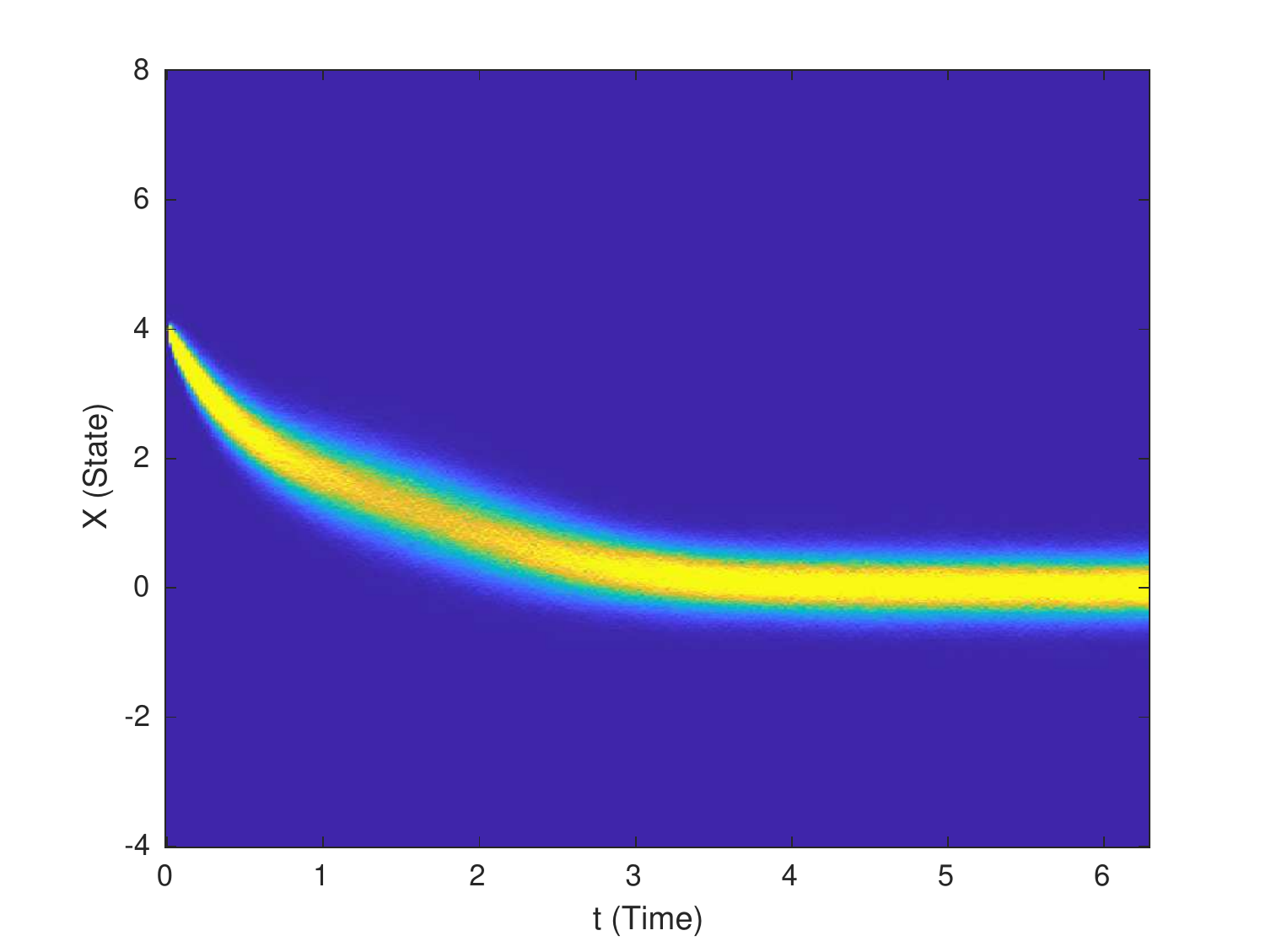}} 
    \subfloat[Parallel-Sampled Suboptimal\label{fig:distsubopt}]{
        \includegraphics[width=0.47\linewidth]{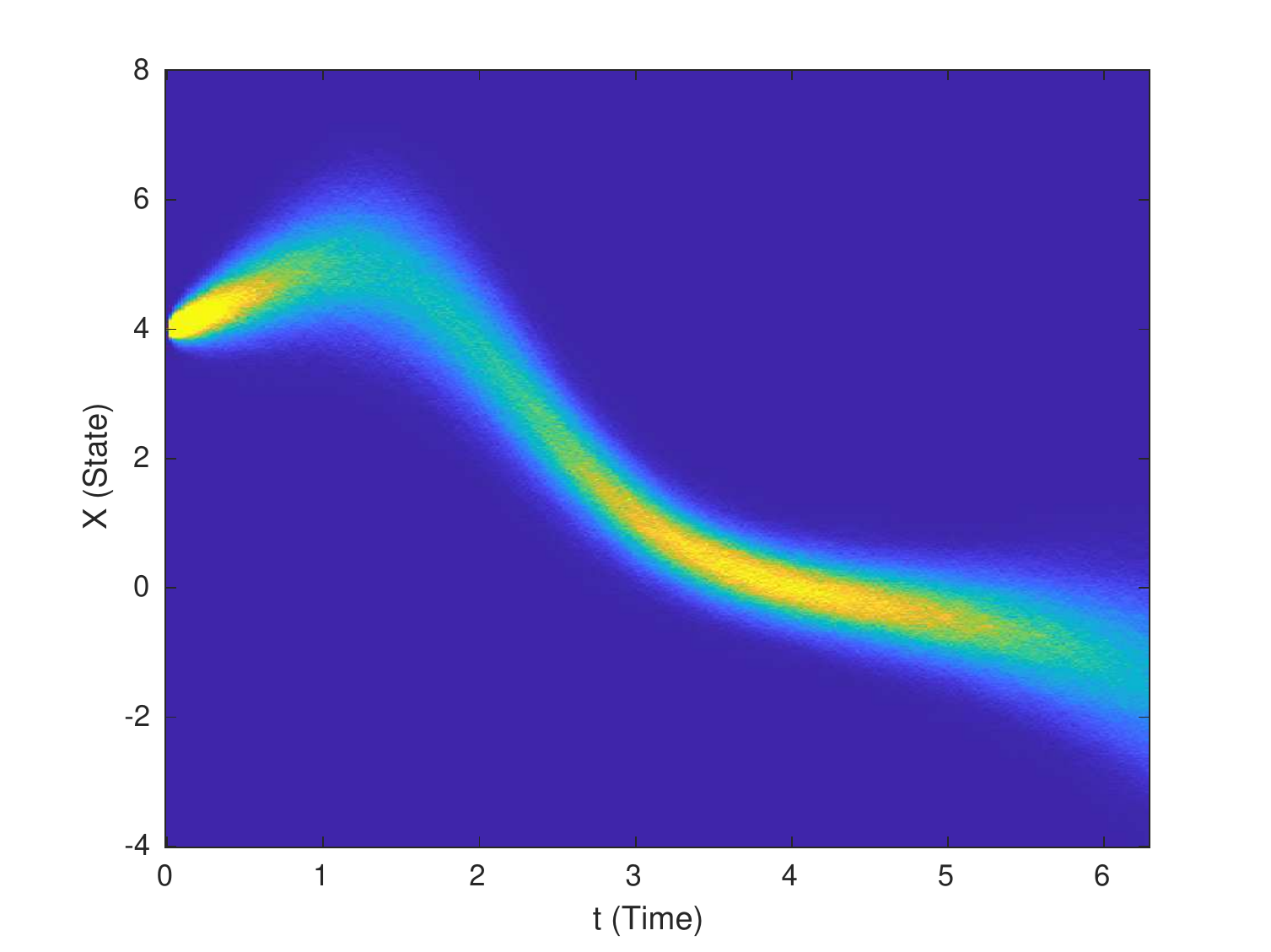}} \\
    \subfloat[RRT-Sampled ($\P$)\label{fig:distrrt}]{
        \includegraphics[width=0.47\linewidth]{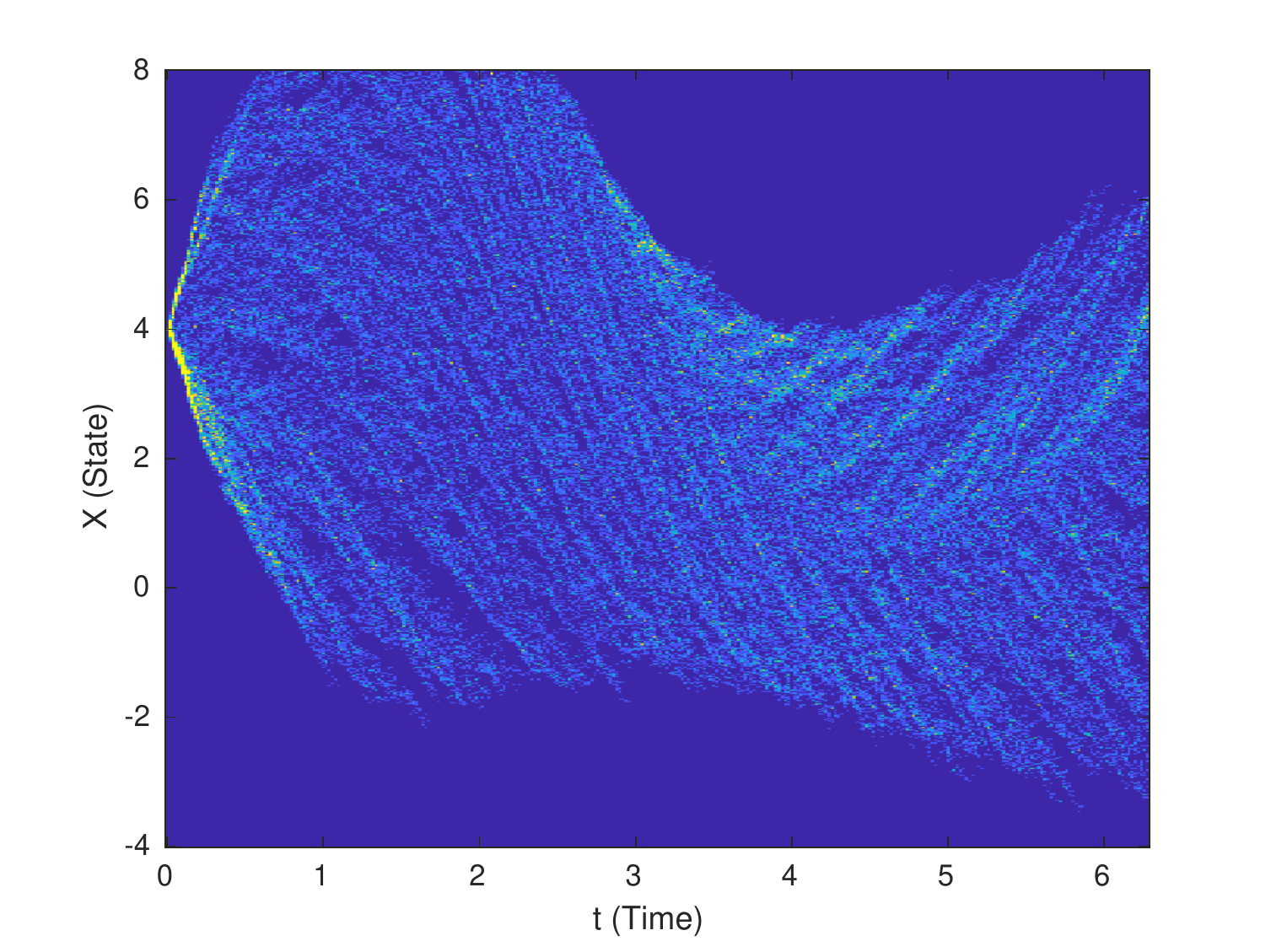}} 
    \subfloat[RRT-Sampled, Weighted ($\pR$)\label{fig:distweight}]{
        \includegraphics[width=0.47\linewidth]{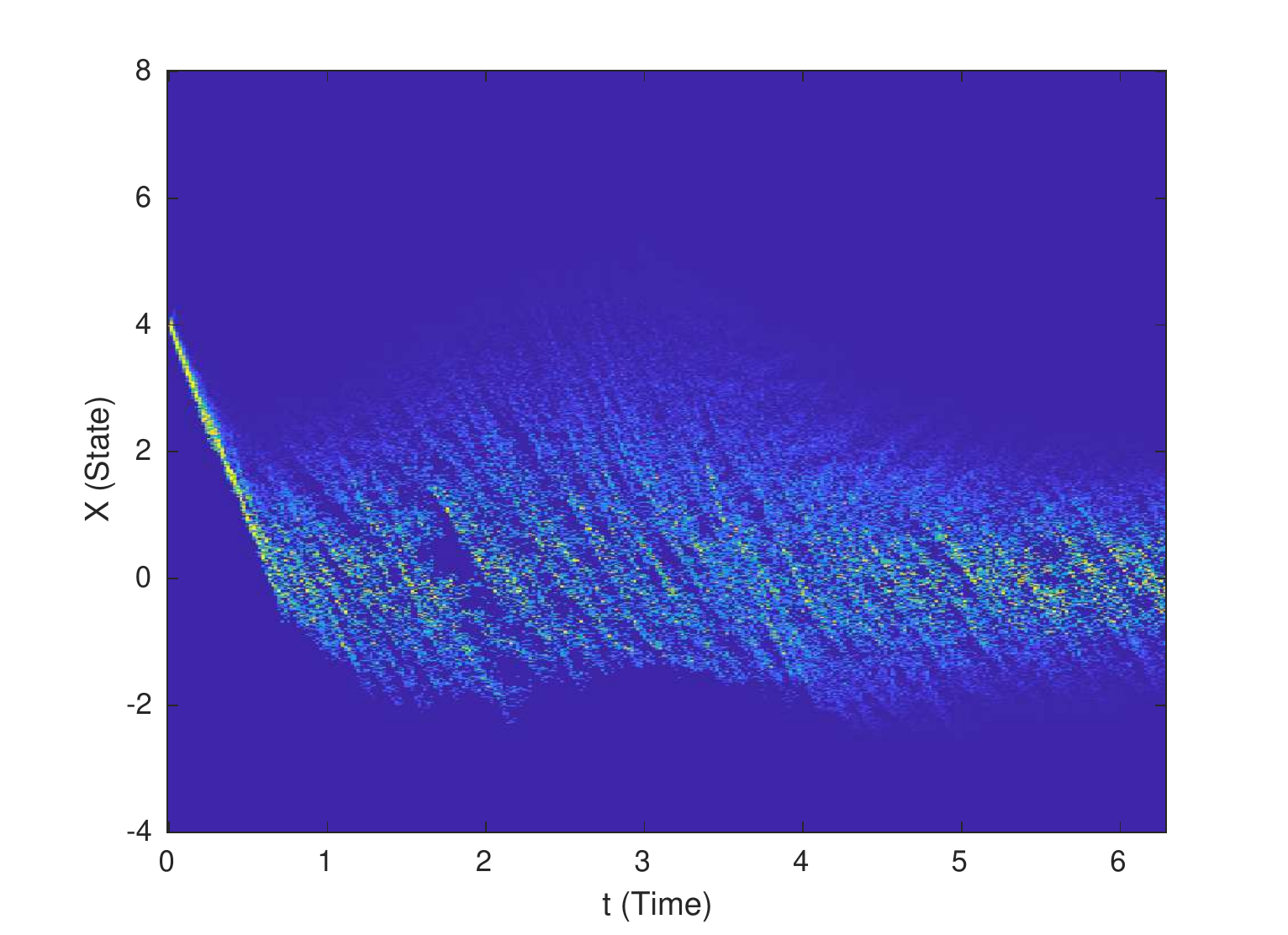}}
    \caption{Heatmap of different measure distributions for a $1$-dimensional SOC problem, 
        illustrating how RRT-sampling and local-entropy weighing can accelerate
        discovery of the optimal distribution.}
    \label{fig:distributions}
\end{figure}

To summarize, in this section we introduced three measures, (a) $\Q$, the measure associated with the target policy $\mu$ for the value function $V^\mu$, (b) $\P$, the sampling measure used in the forward pass to explore the state space, and (c) $\pR_\tau$, the local-entropy weighted measure used in the backward pass to control function approximation accuracy.
Fig.~\ref{fig:distributions} illustrates how these results work together to rapidly discover the optimal distribution. An on-policy method assumes the knowledge of an initial suboptimal control policy, sampled as represented in Fig.~\ref{fig:distributions}~\subref{fig:distsubopt}, and the suboptimal value function is solved in that distribution. This method requires iterative improvement of the policy to produce a distribution which overlaps with the optimal distribution. However, if we begin with a sampling measure which broadly explores the state space as in Fig.~\ref{fig:distributions}~\subref{fig:distrrt}, we can produce an informed heuristic which weighs this distribution as in Fig.~\ref{fig:distributions}~\subref{fig:distweight}, so that the function approximation is concentrated in a near-optimal distribution.
These results leave open the choice for a target policy $\mu$ that produces $\Q$, the drift process $K_s$ that produces $\P$ and the weighing function $\rho_\tau$ that produces $\pR_\tau$. In the following section we propose particular choices for each.

\section{FORWARD-BACKWARD RRT}

In this section, we introduce a numerical method that leverages the continuous-time
    theory of the previous section.
We begin by discussing a generalized approach to approximating the sampling distribution
    $\P$ with a branch-sampling representation.
Next, we introduce FBRRT, an iterative algorithm for solving the SOC problem.
We then propose an RRT-inspired algorithm that leverages the previous theory.
Finally, we propose a heuristic variable $\rho$ for weighing paths.

\subsection{McKean-Markov Branched Sampling}

We approximate the continuous-time sampling distributions with discrete-time 
    McKean-Markov branch sampled paths as presented in \cite{del2013mean}.
First, for a given $\Delta t$, the interval $[0,T]$ is partitioned according to the time steps
    $(t_0 = 0, \ldots, t_i = (\Delta t) i, \ldots, t_N = T)$.
For brevity, we abbreviate $X_{t_i}$ as $X_i$ and similarly for most variables.

In the forward sampling process, we produce a series of path measures $\{\fP_i\}_{i=0}^N$,
\begin{align}
    \fP_i &:=
    \frac{1}{M} \sum_{j=1}^M \delta_{\xi_i^j} \text{,}
\end{align}
    where $\delta$ is the Dirac-delta measure acting on sample paths
\begin{align}
    \xi_i^j &:= (x_{0,i}^j, k_{0,i}^j, 
    x_{1,i}^j, k_{1,i}^j,
    \ldots, k_{i-1,i}^j, x_{i,i}^j) \text{,}
\end{align}
    with $x_{j,i}^j, k_{j,i}^j \in \R^n$.
The path notation $x_{j,i}^j$ indicates that this element is the sample of random variable
    $X_j$ that is the ancestor of sample $x_{i,i}^j$ in path $\xi_i^j$.
Fig.~\ref{fig:parallelvsweighted}~\subref{fig:branched} illustrates how these measures are
    represented using a tree data structure.
Each node in the tree $x^j_i$, alternatively called a particle, 
    is associated with a path $\xi_i^j$ whose final term
    is $x_{i,i}^j = x^j_i$.

The edges in the tree represent an Euler-Maruyama SDE step approximation of the forward
    SDE \eqref{eq:driftfsde}.
When a node in the tree at time $i$ is selected for expansion, it becomes the
    $x^j_{i,i+1}$ element in the path $\xi^j_{i+1}$, its ancestry also included.
The element $k^j_{i,i+1} \sim h(x^j_{i,i+1})$ is sampled from some random function
    which can depend on the state, and, independently, 
    $w^j_{i,i+1} \sim \mathcal{N}(0,\Delta t I_n)$.
The next state in the path is computed as
\begin{align}
    x^j_{i+1,i+1} = x^j_{i,i+1} + k^j_{i,i+1} \Delta t + \sigma(t_i, x^j_{i,i+1}) w^j_{i,i+1}
    \text{.}
\end{align}
    
The measures $\fP_i$ and $\fP_{i+1}$ may not agree on the
    interval $[0,t_i]$.
To see why this is permissible, 
    consider Theorem~\ref{thm:weight} with $\tau = t_{i+1}$ and $t = t_i$.
In a backward step, some $\P_{i+1}$ is used to produce a relationship to solve for the
    deterministic function $V^\mu(t_i,x)$.
But an independent application of the theorem with $\tau = t_i$ and $t = t_{i-1}$
    can use any new measure $\P_i$.
The only requirement is that each $\fP_i$ is consistent with the assumptions placed on $\P_i$.
\begin{figure}
    \centering
    \subfloat[Parallel-Sampled \label{fig:parallel}]{
        \includegraphics[width=0.43\linewidth]{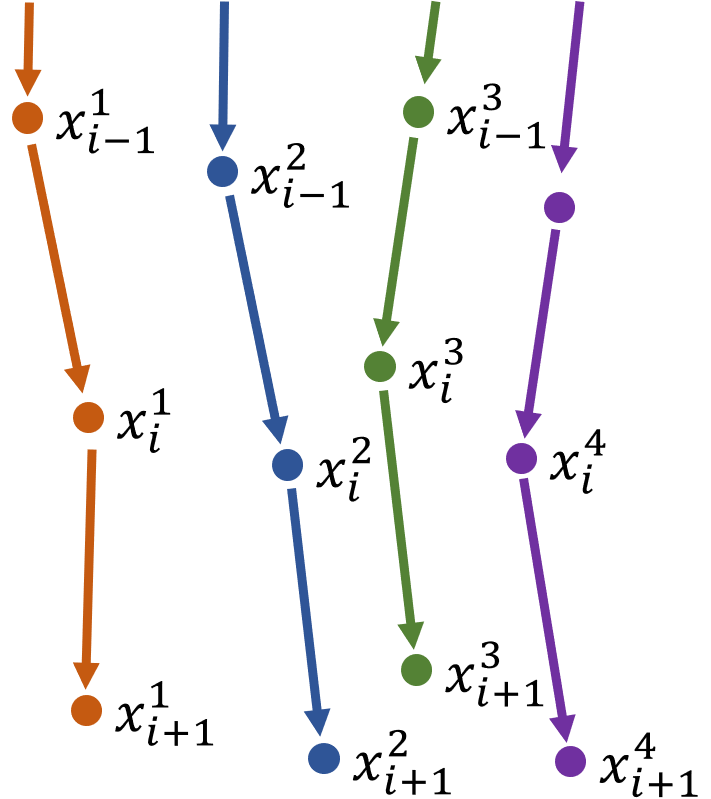}} 
    \subfloat[Branch-Sampled \label{fig:branched}]{
        \includegraphics[width=0.43\linewidth]{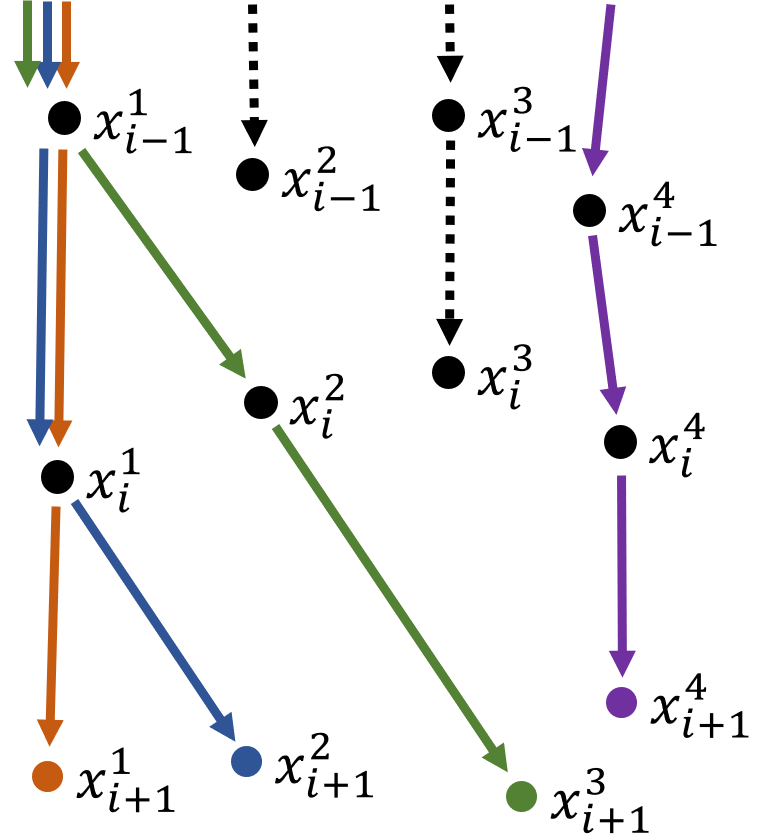}}
        \caption{Comparing parallel-sampling of the path measure $\fP_{i+1}$, 
            in which SDE paths are sampled independently, 
            to the proposed representation.
            Dotted edges are present in the data structure but do not contribute to the path
        measure $\fP_{i+1}$ (but will contribute to $\fP_{i}$ and $\fP_{i-1}$).}
    \label{fig:parallelvsweighted}
\end{figure}

In the construction of $\fP_{i+1}$ in Fig.~\ref{fig:parallelvsweighted}~\subref{fig:branched} 
    we can see that some
    edges are multiply represented in the distribution.
If the drift term $K_i$ were a deterministic function of $X_i$, such a construction
    would represent an unfaithful characterization of the path distribution because
    samples of the Brownian process are independent and thus should be sampled 
    as in Fig.~\ref{fig:parallelvsweighted}~\subref{fig:parallel}.
However, since $K_i$ itself has a distribution, we can interpret overlapping paths
    as the drift having been selected so as to concentrate the paths in a certain part of the
    state space.
The faithful representation of the independent process $W^\P_s$ might be weakened by this
    construction, but some guarantees about the convergence of such measures with increasing
    numbers of samples are available in \cite{del2013mean}.

\subsection{FBRRT Iterative Algorithm}

The goal of the FBRRT algorithm is to produce the set of parameters $\{ \alpha_i \}_{i=1}^N$
    which approximate the optimal value function $V(x;\alpha_i) \approx V^*(t_i, x)$.
The forward pass produces a graph representation $\mathcal{G}$ of the path measures $\{ \fP_i \}_{i=1}^N$.
Given that the optimal policy has the form \eqref{eq:optpoli},
    we define the target policy
\begin{align}
    &\mu_i(x;\alpha_{i+1}) \label{eq:optpolicy2} \\
    &\quad = \argmin_{u \in U} \{ \ell(t_i,x,u) +  f(t_i,x,u)^\top \partial_x V(x;\alpha_{i+1}) \}
    \text{,} \nonumber
\end{align}
    so that it coincides with the optimal control policy
    when the value function approximation is exact.
The backward pass uses $\mathcal{G}$, $\mu_i$, and $\rho_{i+1}$ to produce $\alpha_i$, backwards
    in time.
At the beginning of the next iteration,
    nodes with high heuristic value $\rho_{i+1}$ are pruned from the tree and $\mathcal{G}$ is 
    regrown from those remaining.

\subsection{Kinodynamic RRT Forward Sampling} \label{sec:kinorrt}

In general, we desire sampling methods which seek to explore the whole state space,
    increasing the likelihood of sampling in the proximity of optimal trajectories.
For this reason, we chose methods inspired by kinodynamic RRT, 
    proposed in \cite{lavalle2001randomized}.
The selection procedure for this method ensures that the distribution of the chosen
    particles is more uniformly distributed in a user-supplied region of interest
    $\mathcal{X}^\text{roi} \subseteq \R^n$, more likely to select particles
    which explore empty space, and less likely to oversample dense
    clusters of particles.

With some probability ${\eps^{\text{rrt}}_i \in [0,1]}$ we choose the RRT
    sampling procedure, 
    but otherwise choose a particle uniformly from $\{x_i^{j}\}_{j=1}^M$, 
    each particle with equal weight.
This ensures dense particle clusters will still receive more attention.
Thus, the choice of the parameter $\eps^{\text{rrt}}_i$ balances exploring the state space
    against refining the area around the current distribution.

For drift generation we again choose a random combination of exploration and exploitation.
For exploitation we choose
\begin{align}
    K_i &= f(t_i, X_i, \mu_i(X_i;\alpha_i)) \text{.}
    \label{eq:koptdrift} 
\end{align}
For exploration we choose
\begin{align}
    K_i &= f(t_i, X_i, u^\text{rand}) \text{.}
\end{align}
    where the control is sampled randomly from a user supplied set
    $u^\text{rand} \sim U^\text{rand}$.
For example, for minimum fuel ($L_1$) problems where control is bounded $u \in [-1, 1]$ and
    the running cost is $L = |u|$, we select $U^\text{rand} = \{ -1, 0, 1 \}$
    because the policy \eqref{eq:optpolicy2} 
    is guaranteed to only return values in this discrete set.

Algorithm~\ref{alg:rrtbranchsamp} sketches out the implementation of 
    the RRT-based sampling procedure, producing the forward sampling tree $\mathcal{G}$.
The algorithm takes as input any tree with width $\wt{M}$ 
    and adds nodes at each depth until the
    width is $M$, the parameter indicating the desired width.
On the first iteration there are no value function estimate parameters available
    to exploit,
    so we set $\eps^\text{rrt}=1$ to maximize exploration using the RRT sampling.


\begin{algorithm}
\caption{RRT Branched-Sampling}\label{alg:rrtbranchsamp}
\begin{algorithmic}[1]
    \Procedure{ForwardExpand}{$\mathcal{G},(\alpha_1,\ldots,\alpha_N)$}
    \For{$k = \wt{M}+1,\cdots,M $ } \Comment{Add node each loop}
        \For{$i = 0,\cdots,N-1 $ } \Comment{For each time step}
            \State $\{x^j_i\}_j \gets \mathcal{G}\text{.nodesAtTime}(i)$
            \If{$\eps^\text{rrt} > \kappa^\text{rrt} \sim \text{Uniform}([0,1])$}
                \State $x_i^{\text{rand}} \sim \text{Uniform}(\mathcal{X}^\text{roi})$
                \State $(x_i^{\text{near}},j^{\text{near}}) \gets 
                \text{Nearest}(\{x^j_i\}_j, x_i^{\text{rand}})$
            \Else
                \State $(x_i^{\text{near}},j^{\text{near}}) \sim \text{Uniform}(\{x^j_i\}_j)$
            \EndIf \Comment{$j^{\text{near}}$ is index of selected node}
            \If{$\eps^\text{opt} > \kappa^\text{opt} \sim \text{Uniform}([0,1])$}
                \State $u_i \gets \mu_i(x_i^{\text{near}};\alpha_{i+1})$ 
                \Comment{\eqref{eq:optpolicy2}}
            \Else
                \State $u_i \sim U^\text{rand}$
            \EndIf
            \State $k_i \gets f(t_i, x_i^\text{near}, u_i)$
            \State $w_i \sim \mathcal{N}(0,\Delta t I_{n})$
            \State $x_{i+1}^\text{next} \gets x_i^{\text{near}} + 
            k_i \Delta t + \sigma(t_i,x_i^{\text{near}}) w_i$ 
            \State $j^\text{next} \gets 
            \mathcal{G}\text{.addEdge}(i,j^\text{near}, (x^{\text{near}}_{i}, 
                k_i, x_{i+1}^\text{next}))$
            \State $\overrightarrow{\ell}_{0:i-1} 
                \gets \mathcal{G}\text{.getRunCost}(i-1, j^{\text{near}})$ 
            \State $\overrightarrow{\ell}_{0:i} 
                \gets \overrightarrow{\ell}_{0:i-1} 
                + \ell_i(x_i^{\text{near}}, u_i) \Delta t$ \label{line:pathint}
            \State $\mathcal{G}\text{.setRunCost}(i, j^\text{next}, 
                \overrightarrow{\ell}_{0:i})$
        \EndFor
    \EndFor
    \State \textbf{return} $\mathcal{G}$
    \EndProcedure
\end{algorithmic}
\end{algorithm}

\subsection{Path-Integral Backwards Weighing} \label{sec:locentbw}

We now propose a heuristic design choice for the backward pass weighing variables
    $\rho_{i+1}$, and justify their choice with some theoretical results.
A good heuristic will give high weights to paths likely to have low value over the whole
    interval $[0,T]$.
Thus, in the middle of the interval we care both about the current running cost and the expected cost.
A dynamic programming principle result
    following directly from 
    \cite[Chapter~4, Corollary~7.2]{fleming2006controlled}
    indicates that
\begin{align*}
    &V^*(0,x_0) =  \\
&\quad \min_{u[0,t_{i+1}]} E_{\P_{i+1}^{u}}[\int_{0}^{t_{i+1}} \ell(s,X_s,u_s) \, \ds 
    + V^*(t_{i+1},X_{i+1})] \text{,}
\end{align*}
    where $u_{[0,t_{i+1}]}$ is any control process in $U$ on the interval $[0,t_{i+1}]$ and
    $\P_{i+1}^{u}$ is the measure produced by the drift 
    $K_s = f(s,X_s,u_s)$.
Following this minimization, we choose the heuristic to be
\begin{align}
    \rho_{i+1} = \int_{0}^{t_{i+1}} \ell(s,X_s,u_s) \, \ds 
    + V^*(t_{i+1},X_{i+1}) \text{,}
    \label{eq:rhodef}
\end{align}
    where $u_{[0,t_{i+1}]}$ is chosen identically to how the control for the drift is produced.
Although the theory does not require $K_s$ to be a feasible drift under the dynamic constraints,
    for reasons like this it is useful for it to be chosen in this way.
The running cost is computed in the forward sampling
    in line~\ref{line:pathint} of Algorithm~\ref{alg:rrtbranchsamp}.

\begin{algorithm}
\caption{Local Entropy Weighted LSMC Backward Pass}\label{alg:bptt}
\begin{algorithmic}[1]
    \Procedure{BackwardWLSMC}{$\mathcal{G}$}
    \State $\{\xi^j_N\}_j \gets \mathcal{G}\text{.pathsAtTime}(N)$
    \State $\{x^j_N\}_j \gets \{\xi^j_N\}_j$
    \State $y_N \gets [g(x^1_N) \; \cdots \; g(x^M_N)]^\top$
    \State $\alpha_N \gets 
        \argmin_\alpha \sum_j \Theta_{N} (\widehat{y}^j_{N} - \Phi(x^j_{N}) \alpha)^2$ 
    \For{$i = N-1,\cdots,1$} \Comment{For each time step}
        \State $\{\xi^j_{i+1} \}_j 
        \gets \mathcal{G}\text{.pathsAtTime}(i+1)$
        \For{$j = 1,\cdots,M$} 
            \Comment{For each path}
            \State $(x^j_{i}, k^j_{i}, x^j_{i+1})  
                \gets \xi^j_{i+1}$
            \Comment{$x^j_{i} = x^j_{i,i+1}$, etc.}
            \State $y^j_{i+1} 
                \gets \Phi(x^{j}_{i+1}) \alpha_{i+1}$ 
            \Comment{\eqref{eq:thmcontinfk3}}
            \State $z^j_{i+1} 
                \gets \sigma^\top_{i+1}(x^j_{i+1}) 
                \partial_x \Phi(x^{j}_{i+1}) \alpha_{i+1}$ 
            \Comment{\eqref{eq:yzv2}}
            \State $\mu^j_{i} 
            \gets \mu_{i}(x^j_{i}; \alpha_{i+1})$
            \Comment{\eqref{eq:optpolicy2}}
            \State $d^j_{i} 
            \gets \sigma^{-1}_{i+1}(x^j_{i+1}) (f^\mu_i - k_i^j)$ 
            \State $\wh{y}^j_i \gets 
            y^j_{i+1} + (\ell^\mu_i 
            + z^{j \top}_{i+1} d^j_{i})\Delta t$ 
            \Comment{\eqref{eq:bsdediff2}}
            \State $\overrightarrow{\ell}_{0:i} 
                \gets \mathcal{G}\text{.getRunCost}(i, j)$ 
            \State $\rho^j_{i+1} \gets 
            y^j_{i+1} + \overrightarrow{\ell}_{0:i}$ 
            \Comment{\eqref{eq:rhodef}}
        \EndFor
        \State $\rho_{i+1} \gets \rho_{i+1} 
            - \min_j \{ \rho^j_{i+1} \}$ \label{st:expcond}
        \Comment{$\exp$ conditioning}
        \State $\Theta_{i+1} \gets \exp(-\sfrac{1}{\lambda}\rho_{i+1})$
        \Comment{\eqref{eq:thetadef}}
        \State $\alpha_{i} \gets 
        \argmin_\alpha \sum_j \Theta^j_{i+1} 
        (\widehat{y}^j_{i} - \Phi(x^j_i) \alpha)^2$ 
        \Comment{\eqref{eq:thmcontinfk4}}
    \EndFor
    \State \textbf{return} $(\alpha_1,\ldots,\alpha_N)$
    \EndProcedure
\end{algorithmic}
\end{algorithm}

Algorithm~\ref{alg:bptt} details the implementation of the backward pass with
    local entropy weighting.
The value function is represented by a linear combination 
    of multivariate Chebyshev polynomials up to the 2nd order,
    ${V(x;\alpha_i) = \Phi(x) \alpha_i}$.
Line~\ref{st:expcond} does not, theoretically, have an effect on the optimization,
    since it will come out of the exponential as a constant multiplier,
    but it has the potential to improve the numerical conditioning of the 
    exponential function computation
    as discussed in \cite[Chapter~5, equation (6.33)]{Goodfellow-et-al-2016}.
The $\lambda$ value is, in general, a parameter which must be selected by the user.
For some problems we choose to search over a series of of possible $\lambda$
    parameters, evaluating each one with a backward pass and using the
    one that produces the smallest expected cost over a batch of trajectory rollouts
    executing the computed policy.



\begin{figure}
    \centering
    \subfloat{
        \includegraphics[width=0.97\linewidth,clip]{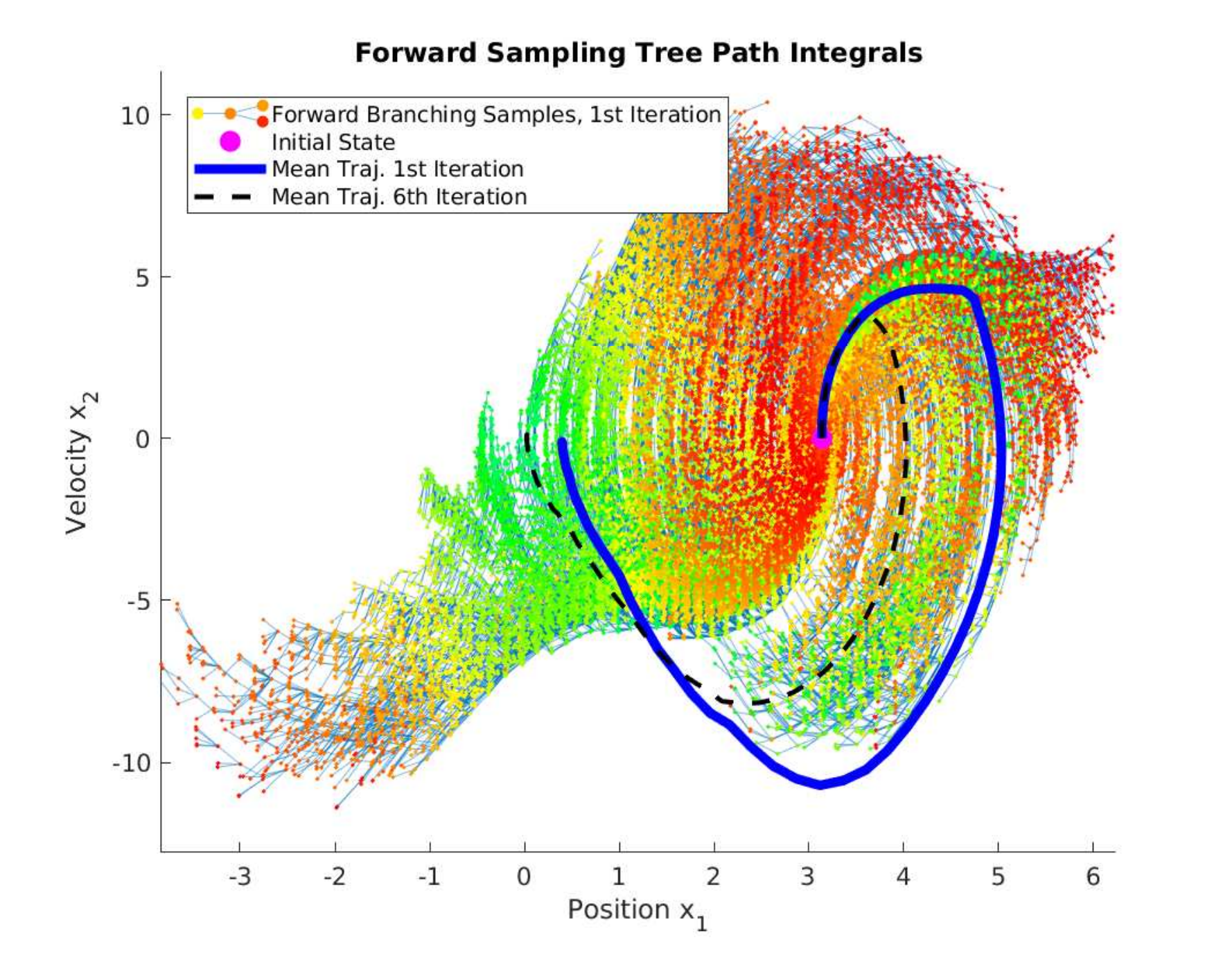}} \\
    \centering
    \subfloat{
        \includegraphics[width=0.97\linewidth,trim=0 0 0 0,clip]{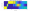}}
\caption{Forward sampling tree for the first iteration of the $L_1$ inverted pendulum problem.
         Hue corresponds to the path-integral heuristic $\rho$ 
         used for weighing
         particles in the backward pass and for pruning the tree (green values are smaller).
         The blue and black dashed lines are the mean of trajectory rollouts,
         following the policies computed at the end of the 1st and 6th iterations respectively.
         Control counts are based on trajectory rollouts
         of the 6th iteration policy computed by FBRRT. The hue of each rectangle indicates
         the relative frequency of each control signal in $\{-1, 0, 1\}$ for each time step.}
\label{fig:fwd_pend_l1}
\end{figure}


\section{NUMERICAL RESULTS}

We evaluated the FBRRT algorithm by applying it to a pair of nonlinear stochastic
    optimal control problems.
For both problems, we used a minimum fuel ($L_1$) running cost of $L(u) = a|u|$, $a>0$,
    $u \in [-1,1]$, where the terminal cost is a quadratic function centered at the origin.
Examples ran in Matlab 2019b on an Intel G4560 CPU with 8GB RAM.

Fig.~\ref{fig:fwd_pend_l1} illustrates our method applied to the $L_1$ inverted
    pendulum problem.
Note that even though there were no paths in the tree that continued along the 
    1st iteration's mean trajectory (blue line) from
    beginning to end, the algorithm was still able to produce a policy in regions
    where no particles were produced.
The green particles along the backward swing inform the policy in the beginning of the trajectory
    while the green particles near the origin inform it near the end, 
    despite taking different paths in the tree.

\begin{figure}
    \centering
    \subfloat[$L_1$ Double Integrator]{
        \includegraphics[width=0.48\linewidth]{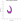}} 
    \subfloat[$L_1$ Inverted Pendulum]{
        \includegraphics[width=0.48\linewidth]{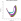}}
    \caption{Trajectory samples from policies generated after the first 6 iterations,
             the first iteration colored red, followed by yellow, green, cyan, dark blue, 
             and magenta. All terminal costs are centered at $(0,0)$.
             Dark thick lines are the mean trajectories.}
    \label{fig:traj_samples}
\end{figure}

The policies computed after the first few iterations 
    are visualized in Fig.~\ref{fig:traj_samples}.
Of significant note is that the policy obtained after only one iteration (red hue)
    does significantly well in general.
For the $L_1$ inverted pendulum problem evaluated in \cite{Exarchos2018}, 
    convergence required $55$ iterations, 
    but for our method only a handful of iterations
    were needed to get comparable performance.
We also compared the convergence speed and robustness of FBRRT 
    to parallel-sampled FBSDE \cite{Exarchos2018} by 
    randomly sampling different starting
    states and evaluating their relative performance over a number of trials.
We normalized the final costs across the initial states by dividing all costs for a particular
    initial state by the largest cost obtained across both methods.
For each iteration, we assign the value of the accumulated minimum value across previous
    iterations for that trial, i.e., the value is the current best cost after running
    that many iterations, regardless of the current cost.
We aggregated these values across initial states and trials into the box plots in
    Fig.~\ref{fig:dblintbatch}.
Since the FBRRT is significantly slower than the FBSDE per iteration due to the RRT nearest
    neighbors calculation,
    we scale each iteration by runtime.
By nearly every comparison, FBRRT converges faster and in fewer iterations than FBSDE, 
    and does so with half as many particle samples.

\begin{figure}
    \centering
    \includegraphics[width=0.48\textwidth]{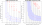}
\caption{Comparison of FBRRT and FBSDE for the $L_1$ double integrator problem for
         random initial states. Expected trajectory costs for the computed policies
         are normalized across different initial conditions.}
\label{fig:dblintbatch}
\end{figure}

\section{CONCLUSIONS AND FUTURE WORK}

In this work, we have proposed a novel generalization of the FBSDE approach to solve
    stochastic optimal control problems, combining both branched sampling techniques
    with weighted least squares function approximation to greatly expand the flexibility
    of these methods.
Leveraging the efficient space-filling properties of RRT methods, we have demonstrated
    that our method significantly improves convergence properties over previous FBSDE methods.
We have shown how the proposed method works hand in hand with a proposed
    path integral-weighted LSMC method, concentrating function approximation in
    the regions where optimal trajectories are most likely to be dense.
We have demonstrated that FBRRT can generate feedback control policies for
    nonlinear stochastic optimal control problems with non-quadratic costs.

Future work includes incorporating modern RRT algorithms, since most
    could be adapted to this approach with the proper book-keeping.
Further, with very minor additions to the forward sampling algorithm, the methods
    might be applied to problems where the system must avoid obstacles,
    though experimental verification of the approach is needed.
Another significant area of research worth investigating is to find better methods of
    value function representation.
Although 2nd-order polynomials generally produce nice policy functions, they are unlikely
    to produce a good approximation of the value function outside of a local region.
Finally, evaluation on higher dimensional problems would be useful to
    demonstrate the usefulness of this method.


\newpage

\bibliographystyle{IEEEtran}
\bibliography{library}

\end{document}